\documentclass[a4paper,reqno,12pt]{amsart}
\usepackage[top=2cm, bottom=2cm, left=3cm, right=3cm]{geometry}
\usepackage[english]{babel}  \usepackage{latexsym,amsfonts} \usepackage{amsmath, amssymb, amsthm} \usepackage{mathrsfs}
\usepackage{enumerate} \usepackage{lmodern} \usepackage{mathpazo}   \usepackage{euscript}  \usepackage[pdftex]{hyperref}
\usepackage{graphicx} \usepackage{rotating}
\usepackage{amsmath}  \usepackage[all]{xy}

\numberwithin{equation}{section} \makeatletter\@addtoreset{equation}{section}

\newcommand{\nus}{\nu}

\newcommand{\C}{\mathbb C}      \newcommand{\R}{\mathbb R}
 
\newcommand{\bz}{\overline{z}} \newcommand{\bxi}{\overline{\xi}}

  \newcommand{\scal}[1]{\left<#1\right>}  
\newtheorem {theorem}{Theorem}[section]            
        \newtheorem {remark}[theorem]{Remark}
\newtheorem {proposition}[theorem]{Proposition}       
         
\newcommand{\Hds}{\mathscr{H}^{2,s}(\C)}
\newcommand{\Ldnus}{\mathcal{L}^{2,\nus}(\C)}
\newcommand{\Xs}{\mathcal{X}_{s}(\C)}
\newcommand{\Xns}{\mathcal{X}_{n,s}(\C)}
\newcommand{\Fdnus}{\mathcal{F}^{2,\nus}(\C)}
\newcommand{\Fdnusn}{\mathcal{F}^{2,\nus}_n(\C)}
\newcommand{\Fdnusz}{\mathcal{F}^{2,\nus}_0(\C)}
\newcount\wang\newsavebox{\wangtext} \newdimen\wangspace
 \def\wheel#1{\savebox{\wangtext}{#1}%
 \wangspace\wd\wangtext\advance\wangspace by
 .5cm%
 \centerline{%
 \rule{0pt}{\wangspace}%
 \rule[-\wangspace]{0pt}{\wangspace}%
 \wang=-180\loop\ifnum\wang<180
 \rlap{\begin{rotate}{\the\wang}%
 \rule{.4cm}{.1pt}#1\end{rotate}}%
 \advance\wang by 20\repeat}}

\usepackage{color}

\definecolor{gris}{gray}{0.5}

\begin{document}


\title[]{The orthogonal complement of the Hilbert space  associated to holomorphic Hermite polynomials}
\author{A. Benahmadi}  \email{abdelhadi.benahmadi@gmail.com}
\author{A. Ghanmi} \email{ag@fsr.ac.ma}
\author{M. Souid El Ainin} \email{msouidelainin@yahoo.fr}
\address{A.G.S.-L.A.M.A., CeReMAR, Department of Mathematics,
	\newline P.O. Box 1014,  Faculty of Sciences,
	\newline Mohammed V University in Rabat, Morocco}

\begin{abstract}
We study the orthogonal complement of the Hilbert subspace considered by by van Eijndhoven and Meyers in \cite{Van1990} and associated to holomorphic Hermite polynomials. 
A polyanalytic orthonormal basis is given and the explicit expressions of the corresponding reproducing kernel functions and Segal--Bargmann integral transforms are provided. 
\end{abstract}

\maketitle

 \section{Introduction}

 The known Hermite polynomials and their different generalizations have been one of the most interesting fields for research, since their introduction by Lagrange and Chebyshev. 
 They appear in a wide spectrum of research domains including  enginery, pure and applied mathematics, and different branches of physics.  The classical ones on the real line $\R$ are defined by (\cite{Rainville71,Szego75,Thangavelu93})
 $$ H_m(x) = (-1)^m e^{x^2} \partial_x^m e^{-x^2} 
 = m!\sum_{k=0}^{[m/2]} \frac{(-1)^k}{k!} \frac{(2x)^{m-2k}}{(m-2k)!}  .$$
 Here and elsewhere after, we use $\partial_x$ to denote the partial differential operator $\partial/\partial_x$. 
 They can be extended to the whole complex plane $\C$ by replacing the real $x$ by the complex variable $z$, leading to the class of holomorphic Hermite polynomials $ H_m(z)$. The last ones inherit the most of the algebraic properties of $H_m(x)$ by analytic continuation. Moreover, they possess further interesting analytic properties.
 The associated functions
    \begin{align}\label{onset}
 \psi^s_m (z) =
 \left( \frac{1-s}{\pi \sqrt{s}} \right)^{1/2} \left(\frac{1-s}{1+s}\right)^{m/2}\frac{e^{-\frac{z^2}{2}}}{\sqrt{2^m m!}}
 H_m(z) ,
 \end{align}
  for given fixed $ 0<s<1$, satisfy the orthogonal property (\cite{Van1990})
  \begin{align}\label{orthRelation}
 \int_{\C}\psi^s_n (z)\overline{ \psi^s_m(z)}e^{-\frac{1-s^2}{2s}|z|^2}e^{\frac{1+s^2}{4s}(z^2+\overline{z}^2)}
 d\lambda(z) = \delta_{n,m},
\end{align}
where $d\lambda(z)=dxdy$ being the Lebesgue measure on $\C\equiv \R^2$. This is to say that the functions $\psi^s_n (z)$ in \eqref{onset} form an orthonormal system in the Hilbert space $\Hds:=L^2(\C,\omega_s d\lambda)$, where the weight function $\omega_s$ is given by $$ \omega_s(z,\bz)
=e^{\frac{1+s^2}{4s}(z^2+\overline{z}^2)-\frac{1-s^2}{2s}|z|^2}
.$$
Equivalently, if $M_\alpha$ denotes the multiplication operator
\begin{align}\label{MultOpMs}
[ M_\alpha f](z) := M_\alpha(z) f(z)=  e^{\frac{1+s^2}{4s} z^2 }f(z) , \quad M_\alpha(z):= e^{\frac{1+s^2}{4s} z^2 },
\end{align}
with $\alpha=\alpha_s =\frac{1+s^2}{4s}$, then the functions
   \begin{align}\label{onset2}
\widetilde{\psi}^s_m (z) = [ M_\alpha \psi^s_m ](z)
\end{align}
 form an orthonormal system in $\Ldnus:= L^2(\C,e^{-\nus |z|^2}d\lambda)$, where
 $ \nu={\nu_s}= \frac{1-s^2}{2s}$.
Accordingly, we define the  Hilbert subspace $\Xs$ as in \cite{Van1990} by $\Xs 
=\mathcal{H}ol\cap \Hds$. Its companion
$\Fdnus
=\mathcal{H}ol\cap \Ldnus
=  M_\alpha (\Xs)$ is the classical Bargmann--Fock space of weight ${\nus}$ (see e.g. \cite{BenahmadiG2018,Folland89}).

The aim of the present paper is three
folds

\begin{enumerate}	
\item[1] Review and complete the study of the space $\Xs$. In particular, we provide the associated Segal--Bargmann transforms for the configuration space $L^2(\R)$. See Section 2.

\item[2] Study 
a Hilbertian decomposition of $\Hds$ in terms of some reproducing kernel Hilbert subspaces $\Xns$, and provide to each $\Xns$ an orthonormal basis generalizing the ones in \eqref{onset2} to the polyanalytic setting, as well as the explicit expression of the reproducing kernel of $\Xns$.  See Section 3.

\item[3]  We also give the corresponding Segal--Bargmann integral transform. See Section 3.


\end{enumerate}


\section{Complements on $\Xs$ }

We begin with the following

\begin{proposition}[\cite{Van1990}]\label{RepKer0}
	The functions $\psi^s_n $ constitute an orthonormal basis of the reproducing kernel Hilbert space  $\Xs$ with kernel given explicitly by
	   \begin{align}\label{RepKern} K^s(z,w)=\frac{1-s^2}{2\pi s} e^{-\frac{1+s^2}{4s}(z^2+\overline{w}^2)+\frac{1-s^2}{2s}z\overline{w}}.
	      \end{align}
\end{proposition}

  
\begin{proof}
   The proof of \eqref{RepKern} can be handled by invoking the unitary operator $M_\alpha$ in \eqref{MultOpMs} and observing that the functions 
   \begin{align}\label{scaledbasis}
   \phi^s_m(z) = \frac{1}{\sqrt{\pi m!}} \left(\frac{1-s^2}{2s}\right)^{(m+1)/2}           
   e^{-\frac{1+s^2}{4s}z^2}z^m \end{align}
   form an orthonormal basis of $\Xs$,
  so that  one concludes for the explicit expression of $ K^s(z,w)$ by performing
     $ K^s(z,w)=\sum\limits_{m=0}^{+\infty} \phi^s_m(z)\overline{\phi^s_m(w)}$ and next using the generating function of the Hermite polynomials $H_n(z)$ (\cite[p. 130]{Rainville71}).
\end{proof}

%


 \begin{remark}\label{RemRepKer}
 	 The expression of the reproducing kernel can also be proved in an easy way by making appeal to the following general principle. Let $\mathcal{H}$ be a separable reproducing kernel Hilbert space (RKHS) on the complex plane and denotes by $K^{\mathcal{H}}$ its reproducing kernel function. If $M$ is a multiplication operator by a function
 	 $M(z):= e^{\psi(z)}$. 
 	 Then,
 	 $\mathcal{H}'= M\mathcal{H}$ is a RKHS whose kernel function is given by
     \begin{align}\label{RemRepKerF}
      K^{\mathcal{H}'}(z,w)=e^{\psi(z)} K^{\mathcal{H}} (z,w) e^{\overline{\psi(w)}}.
 	 \end{align}
 \end{remark}

%
 \begin{remark}
  The space $\mathop{\cup}\limits_{0<s<1}
   \Xs=S^{1/2}_{1/2}$ is the Gelfand--Shilov space (of holomorphic functions) extended to $\C$ (see \cite{Van1990}).
 \end{remark}


In the sequel, we consider the  integral transform of Segal--Bargmann type 
 \begin{align}\label{SBT}
 [\mathscr{B}_s f] (z) := \int_{\R}  B_s(t,z) f(t) dt
  \end{align}
   associated to the kernel function 
        \begin{align}\label{KerFct}
 B_s(t,z) :=
 \left( \frac{1-s^2}{2\pi s\sqrt{s \pi}} \right)^{1/2}  
\exp\left( -  \frac{1}{2s} t^2 -  \frac{1}{2s} z^2
+ \frac{\sqrt{1-s^2}}{s}  tz \right) .
 \end{align}
 Then, we assert

 \begin{theorem} \label{thmSBT}
 	The transform $\mathscr{B}_s$ defines a unitary isometric integral transform from the configuration Hilbert space $L^{2}(\R)$  onto $\Xs$.
 \end{theorem}

 \begin{proof}
The kernel function $ B_s(t,z)$ in \eqref{KerFct} can be rewritten as
       \begin{align}\label{KerFctExp}
B_s(t,z) := \sum_{m=0}^\infty f_m(t)  \psi^s_m (z), 
  \end{align}
  where
   \begin{equation}\label{Orthonbasis}
  f_m(t) =   \frac{e^{-\frac{t^2}2}}{\sqrt{2^m  m!\sqrt{\pi} }}    H_m(t)
  \end{equation}
  is an orthonormal basis of $L^{2}(\R)$.
  Indeed, we have 
   \begin{align*}\label{KerFct}
  \sum_{m=0}^\infty f_m(t)  \psi^s_m (z)
  &=  \left( \frac{1-s}{\pi \sqrt{s \pi}} \right)^{1/2}  
  e^{-\frac{1}{2}(t^2+z^2)}
  \sum_{m=0}^\infty   \left(\frac{1-s}{1+s}\right)^{m/2}
  \frac{H_m(t) H_m(z)}{2^m m!}.
  \end{align*}
   The rest of the proof is straightforward making use of the Mehler formula for the Hermite polynomials extended to the complex plane, to wit (\cite[p.174, Eq. (18)]{Mehler1866}, see also \cite[p.198, Eq. (2)]{Rainville71})
 \begin{equation}\label{MehlerkernelHnsigma}
 \sum_{m=0}^\infty \frac{\lambda^m }{2^m  m!}  H_m (t) H_m (z)
 = \frac{1}{\sqrt{1 - \lambda^2}}  \exp\left( \frac{- \lambda^2 (t^2 + z^2) + 2 \lambda tz  }{1 - \lambda^2}  \right)
 \end{equation}
 valid for every fixed  $0<\lambda<1$.
  \end{proof}

\begin{remark}
	By means of \eqref{SBT} and \eqref{KerFctExp}, we have
$
[\mathscr{B}_s f_m] (z) = \psi^s_m (z).
$ 
Moreover, the inversion formula of $\mathscr{B}_s$ is given by 
$$
[\mathscr{B}_s^{-1} \varphi] (t) = \int_{\C}  \varphi (z)  B_s(t,\bz) \omega_s(z,\bz) d\lambda(z).
$$
\end{remark}

 \begin{remark}
	By considering $\widetilde{B_s}(t,z):=  s^{1/4} B_s(s^{1/2} t,z) $, we define an integral transform $\widetilde{\mathscr{B}}_s$ from  $L^{2}(\R)$  onto $\Xs$ such that $[\widetilde{\mathscr{B}}_s f_mn] (z) = \phi^s_m (z)$, where $\phi^s_m$ are as in \eqref{scaledbasis}, since
	$$\widetilde{B_s}(t,z) = \sum_{m=0}^\infty f_m(t) \phi^s_m(z).$$
\end{remark}

 \section{A special orthonormal basis of $\Hds$}
 	
 	
 	
 The multiplication operator $M_\alpha:  f \longmapsto  M_\alpha f=e^{\alpha z^2}f $ defines a unitary operator from $\Hds$ onto $\Ldnus$. Moreover, it maps isometrically the Hilbert subspace $\Xs$ onto the Bargmann--Fock space $\Fdnus$. Therefore, an orthogonal decomposition of
 $\Hds$ can be deduced easily from the one
 of $\Ldnus$,  
  $ \Ldnus = \bigoplus_{n=0}^\infty\Fdnusn,
  $ 
   given in terms of the polyanalytic Hilbert spaces
  \begin{align*}
\Fdnusn= Ker|_{\Hds} \left( \Delta_{\nu} - n{\nus} Id \right)
   \end{align*}
 where $\Delta_{\nu} := -\partial_z\partial_{\bz}+\nu\bz \partial_{\bz}$ and with $ \Fdnusz=  \Fdnus$.
 See for e.g. \cite{GhIn2005JMP} for details. In fact, the consideration of $\Xns :=M_{-\alpha}\Fdnusn$ leads to the orthogonal decomposition
   \begin{align*} \Hds = \bigoplus_{n=0}^\infty \Xns.
\end{align*}
An immediate orthonormal basis of $\Xns$ is then given by
$ e^{-\alpha z^2} H_{m,n}^{{\nus}}(z,\bz)$
for varying $m,n=0,1,2,\cdots $, where
$$H_{m,n}^{{\nus}}(z,\bz) := (-1)^{m+n} e^{{\nus}|z|^2} \partial_{\bz}^m \partial_z^n \left( e^{-{\nus}|z|^2}\right) $$
 denotes the weighted polyanalytic complex Hermite polynomials \cite{Gh13ITSF,Gh2017,Ito52}, generalizing the monomials ${\nus}^m z^m=H_{m,0}^{{\nus}}(z,\bz)$.

The main aim in this section is to provide another "nontrivial" orthonormal basis $\psi^s_{m,n}(z,\bz)$ of $\Hds$, consisting of polyanalytic functions generalizing $\psi^s_m $ and whose first elements are the holomorphic functions $\psi^s_m (z)$ in \eqref{onset2}, i.e., $\psi^s_{m,0}(z,\bz)=\psi^s_m (z)$.
and obtained an appropriate basis of the space $\Xns$. The introduction of $\Xns$ entails the consideration of the integral transform 
   \begin{align*}
  [\mathscr{W}^s_{n} f] (z,\bz) =   \left(\frac{\nu}{\pi}\right) 
   \left(\frac{\nu^n}{n!}\right)^{1/2} 
  e^{-\alpha z^2}
   \int_{\C} e^{-\nu|\xi|^2 + \alpha \xi^2 + {\nus} \overline{\xi} z }  (\bz-\bxi)^n    \psi(\xi)  d\lambda(\xi).
\end{align*}
Then, we can prove the following 

 \begin{theorem}\label{OrthBasis} The transform $\mathscr{W}^s_{n}$ is a unitary integral transform from $\Xs $ onto $\Xns$. Moreover, the functions 
 	 \begin{align}\label{nesfct}
\psi^s_{m,n}(z,\bz)=
\left( \frac{1-s}{\pi\nu^n n!  \sqrt{s}} \right)^{1/2} 
 \left(\frac{1-s}{1+s}\right)^{m/2} 
\frac{e^{-\frac{z^2}{2}}}{\sqrt{2^m m!} } \left( \nabla_{\nu,\alpha-\frac 12}^n H_{m}\right) (z)  ,	  
\end{align}
 	where $\nabla_{\nu,\alpha} :=  -\partial_z +\nu \bz -2\alpha z$, form an orthonormal basis of $\Xns$.	
 \end{theorem}

  \begin{proof} The proof lies essentially on the observation that  the unitary operator $\mathscr{W}^s_{n}$ can be rewritten as $\mathscr{W}^s_{n} = M_{-\alpha} \mathscr{T}^{\nu}_{0,n} M_{\alpha}$, where $\mathscr{T}^{\nu}_{k,n}$ is the integral transform considered 
  	in \cite[Eq. (2.17)]{BenahmadiG2018} and given by
  	$$
  	[\mathscr{T}^{\nu}_{k,n}\psi](z,\bz)=
  	\left(\frac{(-1)^n \nu}{\pi \sqrt{k! n! \nu^{k+n} }}\right)
  	\int_{\C} e^{-\nu|\xi|^2 + \nu \overline{\xi} z }  H^\nu_{k,n}(\xi-z, \bxi-\bz)    \psi(\xi)  d\lambda(\xi),
  	$$
  	as well as on the fact that $\psi^s_{m,n}(z,\bz):=  [\mathscr{W}^s_{n} \psi^s_m](z,\bz)$.
  Thus, by means of \cite[Theorem 2.12]{BenahmadiG2018}, keeping in mind the fact that the polynomials $H^\nu_{m,n}(z,\bz)=:\nabla_{\nu,0}^n (z^m)$ is an orthogonal basis of $\Ldnus$ \cite{Ito52,Gh13ITSF}, 
 the following 	$$[\mathscr{T}^{\nu}_{0,n}\psi](z,\bz) =  \left(\frac{1}{ \nu^n n!}\right)^{1/2} \nabla_{\nu,0}^{n} \psi, $$
 holds true for every nonnegative integers $n$ and any $\psi \in \Ldnus\cap \mathcal{C}^{n}(\C)$. 
The rest of the second assertion is straightforward since the functions $ \psi^s_m$ form an orthonormal basis of $\Xs$.
 The explicit expression of $\psi^s_{m,n}(z,\bz)$ follows by direct computation. Indeed, we have
  	 \begin{align*}
  	 \psi^s_{m,n}(z,\bz) 
  	 &=  M_{-\alpha} \mathscr{T}^{\nu}_{0,n} M_{\alpha} \psi^s_{m}(z)
  	 \\& =
  	  \left(\frac{1}{ \nu^n n!}\right)^{1/2} M_{-\alpha} \nabla_{\nu,0}^{n} \left( M_{\alpha} \psi^s_{m}\right) (z)\\
  	   &=  
  	  \left(\frac{1}{ \nu^n n!}\right)^{1/2}  \nabla_{\nu,\alpha}^n \left(\psi^s_{m}\right) (z)
\\&=  
\left(\frac{1}{ \nu^n n!}\right)^{1/2} e^{\frac{-z^2}{2}} \nabla_{\nu,\alpha-\frac 12}^n \left(e^{\frac{z^2}{2}} \psi^s_{m}\right) (z)  	  ,
  	  \end{align*}
  	  since 
  	  $\nabla_{\nu,a} \left( M_\gamma \psi\right) = M_\gamma \nabla_{\nu,a+\gamma} \psi$ and  
  	  $ \nabla_{\nu,0}^n \left( M_\gamma \psi\right) = M_\gamma \nabla_{\nu,\gamma}^n \psi$.
 \end{proof}	



 \begin{remark}
The inverse of $ \mathscr{W}^s_{n}:  \Xs \longrightarrow \Xns$ 
is given by $[\mathscr{W}^s_{n}]^{-1}=M_{\alpha} \mathscr{T}^{\nu}_{n,0}M_{-\alpha}$, 
More explicitly
$$ [\mathscr{W}^s_{n}]^{-1}\psi(z) =
\left(  \frac{\nu} {\pi} \right) \left( \frac{\nu^{n}}{n!}\right)^{1/2}   e^{\alpha z^2 }
\int_{\C} e^{-\nu|\xi|^2 -\alpha \xi^2  + \nu \overline{\xi} z } (\xi-z)^n    \psi(\xi)  d\lambda(\xi).
$$
\end{remark}

\begin{remark}
The new class of functions in \eqref{nesfct} generalizes the one studied in \cite{BenahmadiG2019} and the previous theorem provide an integral representation of the special functions $\psi^s_{m,n}(z,\bz)$.  Moreover, it 
is closely connected to the polynomials \begin{align}\label{Datt1997}
H'_{m,n}(x,y;z,w|\tau)=m!n!\sum\limits_{k=0}^{min(n,m)}\frac{(-\tau)^k}{k!}\frac{H'_{n-k}(x,y)}{(n-k)!}\frac{H'_{m-k}(z,w)}{(m-k)!} 
\end{align}
 in \cite{DattoliLorenzuttaMainoTorre1997},
where  $H'_n(x,y):=i^n y^{\frac{n}{2}} H_n\left( \frac{x}{2 i} y^{-\frac{1}{2} }\right) .$
\end{remark} 


 The considered space  $\Xns$ is a reproducing kernel Hilbert space for the point evaluation map in $\Xns$ is continuous.  This, can be recovered easily by means of Remark \ref{RemRepKer}. Thus, we assert

\begin{theorem}\label{ExRepKer}
 The explicit expression of the reproducing kernel of $\Xns$  is given by
  $$  K^s_n(z,w)=\left(  \frac{1-s^2}{2\pi s}\right) \frac{(-1)^n}{n!\nu^n}e^{\nu z\overline{w}-\alpha(z^2+\overline{w}^2)}H^\nu_{n,n}(z-w,\overline{z}-\overline{w}).$$
   \end{theorem}

  \begin{proof} By means of Remark \ref{RemRepKer}, the reproducing kernel $K^s_n(z,w)$ of $\Xns$ obeys \eqref{RemRepKerF}. Hence, we have
  $  	K^s_n(z,w)=M_\alpha(z) K^{\Fdnus} (z,w) \overline{M_\alpha(w)}$.
  	where $K^{\Fdnus_n} $ is the reproducing kernel of the generalized Bargmann space $\Fdnus_n$ given by \cite{GhIn2005JMP}
  	   	\begin{align*}K^{\Fdnus_n}_{n} (z,w)=\left( \frac{\nu}{\pi}\right)  \frac{(-1)^n}{n!\nu^n}e^{\nu z\overline{w}}H^\nu_{n,n}(z-w,\overline{z}-\overline{w}).
  	   	\end{align*}
  \end{proof}

 \begin{remark}
	For $n=0$ we recover the reproducing kernel 
	of the Hilbert space   $\Xs$  in Proposition \ref{RepKer0}. 
\end{remark}

\begin{remark}

The identity
\begin{align*}
H^\nu_{n,n}(z-w,\overline{z}-\overline{w}) = (-1)^n
e^{ \left( \alpha -\frac{1}{2}\right) (z^2+\overline{w}^2) -\nu z\overline{w} } 
\nabla_{\nu,\alpha-\frac 12}^{n_z} \overline{\nabla_{\nu,\alpha-\frac 12}^{n_w}}  
e^{- \left( \alpha-\frac 12\right)   (z^2+\overline{w}^2) +  \nus z \overline{w}  }
\end{align*}
or equivalently 
\begin{align*}
H^\nu_{n,n}(z-w,\overline{z}-\overline{w})= (-1)^n e^{\nu(|z|^2+|w|^2-z\overline{w})}\partial^n_z\partial^n_{\overline{w}}e^{-\nu(|z|^2+|w|^2 + z\overline{w})}
\end{align*}
holds true by comparing the result of Theorem \ref{ExRepKer} to the fact that the reproducing kernel $K^s_n$ can be rewritten as $K^s_n(z,w) = \sum\limits_{m=0}^{+\infty} \psi^s_{m,n}(z)\overline{\psi^s_{m,n}(w)} $, for $\{\psi^s_{m,n}(z,\bz) , m=0,1,2,\cdots\}$, in \eqref{nesfct}, being an orthonormal basis of $\Xns$. 
\end{remark}

We conclude this section by giving the explicit expression of the 
generalized Segal--Bargmann integral transform for the spaces $\Xns$. We have to consider the weighted configuration space $L^{2,\nu}(\R)$ instead of $L^{2}(\R)$, where $\nu >0$. It is the Hilbert space of all square integrable
$\C$-valued functions on $\R$ with respect to the Gaussian measure $e^{-\nu x^2}dx$, for which the rescaled Hermite polynomials 
\begin{align}\label{basis2nu} g^\nu_m(x) 
= \left(\frac{\nu}{\pi}\right)^{\frac{1}{4}} \frac{H_m(\sqrt{\nu}x)}{\sqrt{2^m m!}} 
\end{align}
form an orthonormal basis.   
The associated coherent states transform from $L^{2,\nu}(\R)$ onto $\Xns$ mapping $g^\nu_m$ to $\psi^s_{m,n}$ is given by 
  $$\mathscr{S}^s_n f(z):= 
 \scal{ f,  \overline{S^s_n(.,z)} } _{L^{2,\nu}(\R)} = \int_{\R} f(x) S^s_n(x,z)  e^{-\nu x^2} dx, $$
 where the kernel function $S^s_n(x,z)$ is given by 
  $$ S^s_n(x,z)=\sum\limits_{m=0}^{+\infty}  g^\nu_m(x) \psi^s_{m,n}(z,\overline{z}) .$$
For fixed  $a>0$, $b\in \R$ and $c\in\C$, we define $I_n^{a,b}(z,\bz|c) $ to be the class of polyanalytic polynomials in \cite{BenahmadiG2019},
$$ I_n^{a,b}(z,\overline{z}|c) := (-1)^n e^{a|z|^2 -b z^2 -c z}  \partial_z^n \left( e^{-a|z|^2 + b z^2  + c z}\right) .$$

\begin{theorem}
	We have 
	\begin{align}\label{closed}
	S^s_n(x,z)=
	\left(\frac{\nu}{\pi s}\right)^{\frac{1}{4}}
	\left( \frac{1-s^2}{2\pi s  \nu^n n! }  \right)^{1/2}
	e^{  - \frac{1}{2s} z^2 - \frac{\nu(1-s)}{2s} x^2+ \frac{\nus\sqrt{2s}}{s} x z}
	I^{\nu, -\frac \nus 2 }_n\left( z,\bz \Big | \frac{\nus\sqrt{2s}}{s} x \right) .
	\end{align}
Moreover, the transform $\mathscr{S}^s_n$ 
defines an isomtric transform from  $L^{2,\nu}(\R)$ onto $\Xns$.
\end{theorem}
 
 \begin{proof}
We need only to prove the closed formula \eqref{closed}
for $S^s_n(x,z)$. The rest holds true for general coherent state transformations on the reproducing kernel Hilbert spaces likes $\Xns$. Indeed, starting from \eqref{basis2nu}  and \eqref{nesfct} and applying the Mehler formula \eqref{MehlerkernelHnsigma} the expression of $S^s_n(x,z)$ reduces further to
\begin{align*}
S^s_n(x,z)   
&= 
\left(\frac{\nu}{\pi s}\right)^{\frac{1}{4}}
\left( \frac{1-s^2}{2\pi s  \nu^n n! }  \right)^{1/2}
e^{-\frac{z^2}{2} - \frac{\nu(1-s)}{2s} x^2  }  
\nabla_{\nu,\alpha-\frac 12}^{n_z}  
\exp\left(  
- \frac{1-s}{2s}  z^2 +  \frac{\nus\sqrt{2s}}{s} x z \right)   .
\end{align*}
Using the fact $\nabla_{\nu,\gamma} f = - e^{\nu|z|^2 -\gamma z^2} \partial_z \left( e^{-\nu|z|^2 +\gamma z^2} f\right) $, we get $$\nabla_{\nu,\gamma}^n f = (-1)^n e^{\nu|z|^2 -\gamma z^2} \partial_z^n \left( e^{-\nu|z|^2 +\gamma z^2} f\right) $$
 by induction, and therefore
 \begin{align*}
S^s_n(x,z) 
&= 
\left(\frac{\nu}{\pi s}\right)^{\frac{1}{4}}
\left( \frac{1-s^2}{2\pi s  \nu^n n! }  \right)^{1/2}
e^{ \nu|z|^2  -   \alpha   z^2 - \frac{\nu(1-s)}{2s} x^2} 
(-1)^n  
\partial_z^n \left(  e^{-\nu|z|^2  -\frac \nus 2  z^2 +  \frac{\nus\sqrt{2s}}{s} x z} \right) 
\end{align*}  
Subsequently
\begin{align*}
S^s_n(x,z)=
\left(\frac{\nu}{\pi s}\right)^{\frac{1}{4}}
\left( \frac{1-s^2}{2\pi s  \nu^n n! }  \right)^{1/2}
e^{  - \frac{1}{2s} z^2 - \frac{\nu(1-s)}{2s} x^2+ \frac{\nus\sqrt{2s}}{s} x z}
I^{\nu, -\frac \nus 2 }_n\left( z,\bz \Big | \frac{\nus\sqrt{2s}}{s} x \right) .
\end{align*} 	
 \end{proof}

 \section{Concluding remarks}
  In the previous section the space  $\Xns $ are realized as the image of $\Xs $ the integral transform $\mathcal{W}^s_{n}$ or also as the image of $L^{2,\nu}(\R)$ by the generalized Segal--Bargmann transform $\mathscr{S}^s_n$. Another realization of $\Xns $ is by considering 
 the $n$-th standard Segal--Bargmann transform \cite{BenahmadiG2019}
 $$\mathscr{B}^{\nu}_n\varphi(z)= \frac{\left(\frac{\nu}{\pi}\right)^{\frac{3}{4}}}{\sqrt{2^n\nu^nn!}}\int_{\mathbb{R}}e^{-\nu(x-\frac{z}{\sqrt{2}})^2}H^\nu_n\left( \frac{z+\overline{z}}{\sqrt{2}}-x\right) \varphi(x)dx$$
 from $L^{2,\nu}(\R)$ onto $\Fdnusn$.
 Indeed, one has to deal with $  \mathscr{B}'_{\nu,n}: L^{2,\nu}(\R) \longrightarrow \Xns $, 
 \begin{align*}
 \mathscr{B}'_{\nu,n} f(z,\bz) = \left( M_{-\alpha}\mathscr{B}^{\nu}_n f \right) (z,\bz).
 \end{align*}
  It is clear that for every fixed $b$, the functions $[\mathscr{B}'_{\nu,n}]^{-1}\psi_{m,n}$ form an orthonormal basis of $ L^{2,\nu}(\R)$. But, there is no clear evidence if they are the same or not. We claim that $\left( [\mathscr{B}'_{\nu,n}]^{-1}\psi_{m,n}\right) $ do not depend of $n$. The corresponding Poisson kernel can be given explicitly leading to a nontrivial $1d$-fractional Fourier transform for the Hilbert space $L^{2,\nu}(\R)$.

   
\end{document}